\documentclass[11pt,a4paper]{article}
\usepackage[latin1]{inputenc}
\usepackage{amsmath}
\usepackage{amsthm}
\usepackage{amsfonts}
\usepackage{amsfonts,amsthm,latexsym,amsmath,amssymb,amscd,epsfig,psfrag,enumerate}
\usepackage{graphics,graphicx, bezier, float, color, hyperref}
\usepackage{amssymb,url}
\usepackage{multienum}
\usepackage[table]{xcolor}
\usepackage{multicol,multirow}
\usepackage{graphicx}
\usepackage{fancyvrb}
\usepackage{parskip}
\usepackage[toc,page]{appendix}
\usepackage[top=2.7cm, bottom=2.7cm, left=1.5cm, right=1.5cm]{geometry}
\usepackage{xcolor}

\usepackage{blkarray}
\newtheorem{theorem}{Theorem}[section]
\newtheorem{lemma}{Lemma}[section]

\usepackage[none]{hyphenat}[section]

\numberwithin{equation}{section}
\numberwithin{table}{section}
\numberwithin{figure}{section}

\title{The Discriminant of the Characteristic Polynomial of the $k$th Fibonacci sequence is not a member of the $k$th Lucas sequence}
\author{Herbert Batte$^{1,*} $ and Florian Luca$^{1}$}
\date{}

\begin{document}
\maketitle
\abstract{ Let $k\ge 2$ and $\{L_n^{(k)}\}_{n\geq 2-k}$ be the sequence of $k$-generalized Lucas numbers whose first $k$ terms are $0,\ldots,0,2,1$ and each term afterwards is the sum of the preceding $k$ terms. In this paper, we show that this sequence does not contain the discriminant of its characteristic polynomial.} 

\medskip


\medskip


\medskip

\thanks{$ ^{*} $ Corresponding author}

\section{Introduction}\label{intro}

The study of recurrence sequences plays a fundamental role in number theory. Among the most well-known recurrence relations are those of the Fibonacci sequence and its companion, the sequence of Lucas numbers, both of which have been extensively generalized. Given an integer $k \geq 2$, the $k$-generalized Fibonacci sequence $\{F_n^{(k)}\}_{n\in\mathbb{Z}}$, is defined by the recurrence relation  
\begin{equation*}
	F_n^{(k)} = F_{n-1}^{(k)} + F_{n-2}^{(k)} + \cdots + F_{n-k}^{(k)}, \quad \text{for all} \ n \ge 2,
\end{equation*}
with initial values \( F_{2-k}^{(k)} = \cdots = F_{-1}^{(k)} = F_{0}^{(k)} = 0 \) and $F_1^{(k)} = 1$. This sequence extends the classical Fibonacci sequence which occurs when $k=2$ and has been studied in various algebraic and combinatorial contexts. A closely related sequence is the sequence of $k$-generalized Lucas numbers $\{L_n^{(k)}\}_{n\in\mathbb{Z}}$, which follows the same recurrence  
\begin{equation*}
	L_n^{(k)} = L_{n-1}^{(k)} + L_{n-2}^{(k)} + \cdots + L_{n-k}^{(k)},\quad \text{for all} \ n \ge 2,
\end{equation*}
but with different initial conditions: $L_{2-k}^{(k)} = \cdots = L_{-1}^{(k)} = 0,~L_0^{(k)}=2,~L_1^{(k)}=1$.  The classical Lucas numbers correspond to $k=2$ and have also been extensively investigated in number theory, cryptography, and primality testing. The above $k$-generalized versions offer a rich structure for further exploration.  

Recently in \cite{luca2024}, Luca studied the discriminant of the characteristic polynomial of the $k$-generalized Fibonacci sequence, given by  
\begin{equation*}
	g_k(X) = X^k - X^{k-1} - \cdots - X - 1,
\end{equation*}
which has the formula
\begin{equation*}
	\text{Disc}\left(g_k\right) = \left(-1\right)^{\binom{k+1}{2}-1}\left(\dfrac{2^{k+1}k^k-(k+1)^{k+1}}{(k-1)^2}\right).
\end{equation*}
He showed that the absolute value $|\text{Disc}(g_k)|$ belongs to the sequence $\{F_n^{(k)}\}_{n\geq 0}$ only for $k = 2,3$. This result naturally raises the question of whether some similar property holds for the sequence of $k$-generalized Lucas numbers, since these two sequences share the same characteristic polynomial.  

In this paper, we prove the following result.

\begin{theorem}
\label{thm1.1} 
There are no integer solutions $n\ge 0$ and $k\ge 2$ to the Diophantine equation
$$
L_n^{(k)}=|{\text{\rm Disc}}(g_k)|.	
$$
\end{theorem}

\section{Some useful formulas and inequalities involving $L_n^{(k)}$}

In this section we collect some useful formulas and inequalities concerning $k$-generalized Lucas numbers. First, it is known that the formula
\begin{align*}
	L_n^{(k)} = 3 \cdot 2^{n-2}\quad {\text{\rm holds~for}}\quad 2\le n\le k.
\end{align*}
Furthermore, 
\[
L_{k+1}^{(k)} = 3 \cdot 2^{k-1} - 2.
\]
Using induction proves easily that the inequality
\begin{equation*}
	L_n^{(k)} < 3 \cdot 2^{n-2}\qquad {\text{\rm holds~for~all}}\qquad n\ge k+1.
\end{equation*}
The $k$-generalized Lucas numbers can be expressed in terms of the $k$-generalized Fibonacci numbers via the formula
\begin{equation}
\label{eq:FL}
L_n^{(k)}=2F_{n+1}^{(k)}-F_n^{(k)}.
\end{equation}
To prove the above formula, it is enough to first notice that it holds for $n=-(k-2),\ldots,0,1,2$ and then to invoke the fact that both sequences $\{L_n^{(k)}\}_{n\in {\mathbb Z}}$ and $\{2F_{n+1}^{(k)}-F_n^{(k)}\}_{n\in {\mathbb Z}}$ 
satisfy the same $k$th order linear recurrence to conclude that the above formula holds for all $n\in {\mathbb Z}$. Another usefull formula is 
$$
L_{n}^{(k)}=2L_{n-1}^{(k)}-L_{n-(k+1)}^{(k)}
$$
for all $n\in {\mathbb Z}$ (to see why it holds, replace the left--hand side above by $L_{n-1}^{(k)}+L_{n-2}^{(k)}+\cdots+L_{n-k}^{(k)}$, cancel a term $L_{n-1}^{(k)}$ on both sides of the resulting equation, and recognise that 
the resulting formula is just the formula representing $L_{n-1}^{(k)}$ as the sum of the previous $k$ terms of the sequence).  Reducing the above formula modulo $2$ we get
$$
L_n^{(k)}\equiv L_{n-(k+1)}^{(k)} \pmod 2,
$$
showing that $\{L_n^{(k)}\}_{n\in {\mathbb Z}}$ is periodic modulo $2$ with period $k+1$. This fact is so useful that we record it as a lemma.

\begin{lemma}
\label{lem:1}
The sequence $\{L_n^{(k)}\}_{n\in {\mathbb Z}}$ is periodic modulo $2$ with period $k+1$.
\end{lemma} 

Now let's recall some more formulas related to the sequence of $k$-generalized Lucas numbers. It's characteristic  polynomial
$$
g_k(X) = X^k - X^{k-1} - \cdots - X - 1,
$$
is irreducible in $\mathbb{Q}[X]$. This polynomial has exactly one real root greater than 1, denoted by $\alpha := \alpha(k)$, and all its other roots lie inside the unit circle in the complex plane. Further, this real root lies within the range
\begin{align}\label{eq2.3}
	2\left(1 - 2^{-k} \right) < \alpha < 2,\qquad {\text{\rm for~all}}\qquad k\ge 2.
\end{align}
If we need to refer to all the roots of $g_k(X)$, then we label them as $\alpha_1,\ldots,\alpha_k$ with the convention that $\alpha:=\alpha_1$. For all $n \geq 0$ and $k \geq 2$, it was proved in \cite{BRL} that the $k$-generalized Lucas numbers are bounded by powers of $\alpha$ as
\begin{align}\label{eq2.4}
	\alpha^{n-1} \leq L_n^{(k)} \leq 2\alpha^n.
\end{align}
Next, we define the function
\begin{equation*}
	f_k(x) := \frac{x - 1}{2 + (k + 1)(x - 2)},\qquad x\in {\mathbb C}\backslash \left\{2-{2}/{(k+1)}\right\}.
\end{equation*}
The Binet formula for the general term of the $k$-generalized Lucas  sequence is given by
\begin{equation*}
L_n^{(k)} = \sum_{i=1}^k (2\alpha_i - 1)f_k(\alpha_i)\alpha_i^{n - 1},\qquad {\text{\rm for~all}}\quad n\in {\mathbb Z}.
\end{equation*}
Since $\alpha_2,\ldots,\alpha_k$ are inside the unit circle one would guess that only the first term $(2\alpha_1-1)f_k(\alpha_1)\alpha_1^{n-1}$ already approximates $L_n^{(k)}$ and indeed the approximation
\begin{align}\label{lk_b}
\left|L_n^{(k)} - f_k(\alpha)(2\alpha - 1)\alpha^{n - 1}\right| < \frac{3}{2},\qquad {\text{\rm for~all}}\quad n\ge 2-k
\end{align}
appears, for example, in \cite{BRL}. This tells us that most of the size of $L_n^{(k)}$ comes from the dominant term involving the real root $\alpha$, while the other terms contribute very little. A better estimate than \eqref{lk_b} appears in Section  3.3 page 14 of \cite{bat}, but with a more restricted range of $n$ in terms of $k$. It states that 
\begin{align}\label{lk_b1}
	\left| f_k(\alpha)(2\alpha - 1)\alpha^{n - 1}-3\cdot 2^{n-2}\right| < 3\cdot 2^{n-2}\cdot \frac{36}{2^{k/2}}\qquad {\text{\rm provided}}\qquad n<2^{k/2}.
\end{align}
Another useful combinatorial formula for the $k$th generalized Fibonacci numbers is  due to Cooper and Howard in \cite{cop} and states that
\begin{align}\label{comb}
F_n^{(k)} = 2^{n - 2} + \sum_{j = 1}^{\left\lfloor \frac{n + k}{k + 1} \right\rfloor - 1} C_{n, j} \, 2^{n - j(k+1) - 2}, \quad \text{where} \quad 
C_{n,j} := (-1)^j \left( \binom{n - jk}{j} - \binom{n - jk - 2}{j - 2} \right).
\end{align}
In Equation \eqref{comb}, the assumption that $\displaystyle\binom{n}{m}=0$ if either $n<m$ or one of $n$ or $m$ is negative, holds.
Via equation \eqref{eq:FL} it yields the following formula for the $k$th generalized Lucas number:
 \begin{align}\label{2adic2}
	L_n^{(k)} &= 2F_{n+1}^{(k)}-F_n^{(k)}\nonumber\\
	&= 2\left(2^{n-1}+
	\sum_{j = 1}^{\left\lfloor \frac{n +1+ k}{k + 1} \right\rfloor - 1} C_{n+1, j} \, 2^{n - j(k+1) - 1}\right)-
	2^{n - 2} - \sum_{j = 1}^{\left\lfloor \frac{n + k}{k + 1} \right\rfloor - 1} C_{n, j} \, 2^{n - j(k+1) - 2}\nonumber\\
	&= 3\cdot2^{n-2}+
	\sum_{j = 1}^{\left\lfloor \frac{n +1+ k}{k + 1} \right\rfloor - 1}4 C_{n+1, j} \, 2^{n - j(k+1) -2}
	- \sum_{j = 1}^{\left\lfloor \frac{n + k}{k + 1} \right\rfloor - 1} C_{n, j} \, 2^{n - j(k+1) - 2}.
\end{align}

\section{The $2$-adic valuation of $L_n^{(k)}$}

Since $\{L_n^{(k)}\}_{n\in {\mathbb Z}}$ is periodic modulo $2$ with period $k+1$ (see Lemma \ref{lem:1}), it follows that the $2$-adic valuation of $L_n^{(k)}$ depends on the residue of $n$ modulo $k+1$. So, let us write 
$n=r+m(k+1)$ for some integers $m\ge 0$ and $r\in \{0,1,\ldots,k\}$. 

\begin{lemma}
\label{lem:2adic}
Let $k\ge 2$ and  $n=r+m(k+1)$ with integers $m\ge 0$ and $r\in \{0,1,\ldots,k\}$. The following congruences hold.
\begin{itemize}
\item[(i)] If $r=0$, then 
\begin{equation}
\label{eq:r=0}
L_n^{(k)}\equiv 2(-1)^m \pmod {2^{k-2}}.
\end{equation}
\item[(ii)] If $r=1$, then 
\begin{equation}
\label{eq:r=1}
L_n^{(k)}\equiv (4m+1)(-1)^m\pmod {2^{k-1}}.
\end{equation}
\item[(iii)] If $r=2$, then 
\begin{equation}
\label{eq:r=2}
L_n^{(k)}\equiv (4m^2+6m+3)(-1)^m\pmod {2^{k}}.
\end{equation}
\item[(iv)] If $r\ge 3$, then 
\begin{equation}
\label{eq:rge3}
L_n^{(k)}\equiv (-1)^m 2^{r-2} \left(4\left(\binom{m+r+1}{m}-\binom{m+r-1}{m-2}\right)-\left(\binom{m+r}{m}-\binom{m+r-2}{m-2}\right)\right)\pmod {2^{k+r-2}}.
\end{equation}
\end{itemize}
\end{lemma}

\begin{proof}
Note that 
$$
\left\lfloor \frac{n+1+k}{k+1}\right\rfloor=\left\lfloor \frac{r+(m+1)(k+1)}{k+1}\right\rfloor=m+1,
$$
while
$$
\left\lfloor \frac{n+k}{k+1}\right\rfloor=\left\lfloor \frac{(r+k)+m(k+1)}{k+1}\right\rfloor=\left\{\begin{matrix} m+1 & {\text{\rm if}} & r\ge 1,\\ m & {\text{\rm if}} & r=0. \end{matrix}\right.
$$
Furthermore,  for $j$ in the summation ranges in \eqref{2adic2}, we have 
$$
n-j(k+1)-2=(r-2)+(m-j)(k+1)\ge k+r-2\qquad {\text{\rm except~when}}\qquad j=m.
$$
The case $j=m$ happens in the last term of the first summation in \eqref{2adic2} and also in the last term of the second summation but only when $r\ge 1$. So, when $r=0$, assuming also that 
$m\ge 1$ (otherwise $n=0$ so $L_n^{(k)}=L_0^{(k)}=2$ and congruence \eqref{eq:r=0} holds), we have that 
\begin{eqnarray*}
L_n^{(k)} & \equiv &  2^{0-2}\cdot 4 C_{n+1,m} \pmod {2^{k+0-2}}\equiv (-1)^m \left(\binom{n+1-mk}{m}-\binom{n+1-mk-2}{m-2}\right)\pmod {2^{k-2}}\\
& \equiv &  (-1)^m \left(\binom{1+m(k+1)-mk}{m}-\binom{1+m(k+1)-mk-2}{m-2}\right)\pmod {2^{k-2}}\\
& \equiv &  (-1)^m \left(\binom{m+1}{m}-\binom{m-1}{m-2}\right) \equiv  2(-1)^m \pmod {2^{k-2}}.
\end{eqnarray*}
This proves (i). From now on, $r\ge 1$ so both sums appearing in \eqref{2adic2} have the same length $m$. It follows that 
$$
L_n^{(k)}\equiv 2^{r-2}\left(4C_{n+1,m}-C_{n,m}\right)\pmod {2^{k+r-2}}.
$$
We treat the cases $r=1,~r=2$ and $r\ge 3$ separately. For $r=1$, if $m=0$ then $L_n^{(k)}=L_1^{(k)}=1$ and congruence \eqref{eq:r=1} holds. For $m\ge 1$, we get
\begin{eqnarray*}
L_n^{(k)} & \equiv & 2(-1)^m \left(\binom{n+1-mk}{m}-\binom{n-1-mk}{m}\right)-2^{-1}(-1)^m \left(\binom{n-mk}{m}-\binom{n-2-mk}{m-2}\right)\pmod {2^{k-1}}\\
& \equiv &  
2(-1)^m \left(\binom{2+m(k+1)-mk}{m}-\binom{m(k+1)-mk}{m-2}\right)\\ 
& -& 2^{-1}(-1)^m \left(\binom{1+m(k+1)-mk}{m}-\binom{-1+m(k+1)-mk}{m-2}\right)\pmod {2^{k-1}}\\
& \equiv &  
2(-1)^m \left(\binom{m+2}{m}-\binom{m}{m-2}\right)-2^{-1}(-1)^m\left(\binom{m+1}{m}-\binom{m-1}{m-2}\right)\pmod {2^{k-1}}.\\
\end{eqnarray*}
When $m=1$, we get 
$$
-2\binom{3}{1}+2^{-1}\binom{2}{1}\equiv -5\pmod {2^{k-1}}\equiv (-1)^1(4\cdot 1+1)\pmod {2^{k-1}},
$$
while when $m\ge 2$ we get
\begin{eqnarray*}
L_n^{(k)} & \equiv &  2(-1)^m \left(\binom{m+2}{2}-\binom{m}{2}\right)-2^{-1}(-1)^m\left(\binom{m+1}{1}-\binom{m-1}{1}\right)\pmod {2^{k-1}}\\
& \equiv &   2(-1)^m \left(\frac{(m+2)(m+1)}{2}-\frac{m(m-1)}{2}\right)-(-1)^m\pmod {2^{k-1}}\equiv (-1)^m(4m+1)\pmod {2^{k-1}}.
\end{eqnarray*}
This takes care of the congruence \eqref{eq:r=1} in (ii). For (iii), let $r=2$. When $m=0$, we have $n=2$, so $L_n^{(k)}=3=(-1)^0(4\cdot 0^2+6\cdot 0+3)$. When $m\ge 1$, we get
\begin{eqnarray*}
L_n^{(k)} & \equiv & 2^2(-1)^m \left(\binom{n+1-mk}{m}-\binom{n-1-mk}{m}\right)-2^{-1}(-1)^m \left(\binom{n-mk}{m}-\binom{n-2-mk}{m-2}\right)\pmod {2^{k}}\\
& \equiv &  
2(-1)^m \left(\binom{3+m(k+1)-mk}{m}-\binom{1+m(k+1)-mk}{m-2}\right)\\ 
& -& (-1)^m \left(\binom{2+m(k+1)-mk}{m}-\binom{m(k+1)-mk}{m-2}\right)\pmod {2^{k}}\\
& \equiv &  
4(-1)^m \left(\binom{m+3}{m}-\binom{m+1}{m-2}\right)-(-1)^m\left(\binom{m+2}{m}-\binom{m}{m-2}\right)\pmod {2^{k}}.\\
\end{eqnarray*}
When $m=1$, we get 
$$
-4\binom{4}{1}+\binom{3}{1}\equiv -13\pmod {2^{k-1}}\equiv (-1)^1(4\cdot 1^2+6\cdot 1+3)\pmod {2^{k}},
$$
while when $m\ge 2$ we get
\begin{eqnarray*}
L_n^{(k)} & \equiv &  4(-1)^m \left(\binom{m+3}{3}-\binom{m+1}{3}\right)-(-1)^m\left(\binom{m+2}{2}-\binom{m}{2}\right)\pmod {2^{k}}\\
& \equiv &  
(-1)^m(4m^2+6m+3)\pmod {2^{k}}.
\end{eqnarray*}
Part (iv) is similar, that is, let $r\ge 3$ and $m\ge 0$. Then \eqref{2adic2} gives
\begin{eqnarray*}
	L_n^{(k)} & = & 3\cdot2^{n-2}+
	\sum_{j = 1}^{m}4 C_{n+1, j} \, 2^{n - j(k+1) -2}
	- \sum_{j = 1}^{m} C_{n, j} \, 2^{n - j(k+1) - 2}\\
	& = & 3\cdot2^{n-2}+
	\sum_{j = 1}^{m}\left(4 C_{n+1, j}-C_{n, j}\right) \, 2^{n - j(k+1) -2}\\
	&\equiv & (-1)^m 2^{r-2} \left(4\left(\binom{m+r+1}{m}-\binom{m+r-1}{m-2}\right)-\left(\binom{m+r}{m}-\binom{m+r-2}{m-2}\right)\right)\pmod {2^{k+r-2}}.
\end{eqnarray*}
\end{proof}

\section{The $2$-adic valuation of ${\text{\rm Disc}}(g_k(X))$}

The following result appears in \cite{luca2024}.

\begin{lemma}
\label{lem:disc}
Let $k\ge 2$ and put 
$$
\Delta_k:=\frac{2^{k+1}k^k-(k+1)^{k+1}}{(k-1)^2}.
$$
Then 
$$
\nu_2(\Delta_k)=\left\{ \begin{matrix} 0 & {\text{\rm if}} & k\equiv 0\pmod 2,\\
k-1 & {\text{\rm if}} & k\equiv 1\pmod 2.\end{matrix}\right.
$$
\end{lemma}

\section{Linear forms in logarithms}
To estimate how small a nonzero linear combination of logarithms of algebraic numbers can be, we apply a result known as a Baker-type lower bound. While several such bounds are known in the literature, we make use of a version due to Matveev, as presented in \cite{MAT}. Since in our application we only work with forms in logarithms of positive integers, we bypass the regular prerequisites concerning heights and 
present the inequality that interests us.

\begin{theorem}[Matveev, \cite{MAT}]
	\label{thm:Mat} 
	Let $\gamma_1,\ldots,\gamma_t$ be positive integers larger than $1$. Put $\Gamma:=\gamma_1^{b_1}\cdots \gamma_t^{b_t}-1$ and assume $\Gamma\ne 0$. Then 
	$$
	\log |\Gamma|>-1.4\cdot 30^{t+3}\cdot t^{4.5} (1+\log B)A_1\cdots A_t,
	$$
	where $B\ge \max\{|b_1|,\ldots,|b_t|\}$ and $A_i:=\log \gamma_i$ for $i=1,\ldots,t$.
\end{theorem}

We also need a result due to Bugeaud and Laurent \cite[Th\'eor\`eme 3]{BL}):

\begin{lemma} [Bugeaud and Laurent,  \cite{BL}] \label{lem:BL}
Let $b_1,b_2$ be positive integers and suppose that $\alpha_1$ and $\alpha_2$ are multiplicatively independent positive integers. Put
$\log A_i\ge \max\{\log |\alpha_i|,\log p\}$  for $i=1,2$, and
\begin{equation*}\label{eq:bprime}
b':=\frac{b_1}{\log{A_2}}+\frac{b_2}{\log{A_1}}.
\end{equation*}
Then we have
\begin{equation*}\label{eq:lowpadicbound}
\nu_p(\alpha_1^{b_1}-\alpha_2^{b_2})\leq \frac{24 p g}{(p-1)(\log
p)^4}\,B^2\,\log{A_1}\,\log{A_2} ,
\end{equation*}
where $g$ is the smallest positive integer for which $\alpha_i^g\equiv 1\pmod p$
and
\begin{equation*}\label{eq:B}
B:=\max\left\{\log b'+\log\log p+0.4,10,10\log p\right\}.
\end{equation*}
\end{lemma}

Python is used to perform all computations in this work.

\section{The proof of the main theorem}

\subsection{Bounding $n$ in terms of $k$}

We look at the Diophantine equation
\begin{align}\label{main}
	L_n^{(k)} = \dfrac{2^{k+1}k^k-(k+1)^{k+1}}{(k-1)^2},
\end{align}
and assume that $n\ge 0$ and $k\ge 2$. If we combine \eqref{eq2.4} and \eqref{main}, then
\begin{align}\label{nbd1}
	\alpha^{n-1} \le L_n^{(k)} = \dfrac{2^{k+1}k^k-(k+1)^{k+1}}{(k-1)^2} < \dfrac{2^{k+1}k^k}{(k-1)^2}.
\end{align}
Since $\alpha > 1.6$ for all $k\ge 2$ (by \eqref{eq2.3}), then inequality \eqref{nbd1} tells us that whenever $k\le 200$, we have that $n\le 2529$. We wrote a simple program to check for solutions to \eqref{main} in the ranges $2\le k\le 200$ and $0\le n\le 2529$, and found none. From now on, we assume $k>200$.

Taking logarithms on both sides of \eqref{nbd1}, we get
\begin{align*}
	n - 1 
	&< \frac{(k+1)\log 2 + k\log k - 2\log(k-1)}{\log \alpha}\\
	&< \frac{(k+1)\log 2 + k\log k - 2\log(k-1)}{\log 2 + \log(1 - 2^{-k})},\qquad\text{by \eqref{eq2.3},} \\
	&< \frac{(k+1)\log 2 + k\log k - 2\log(k-1)}{\log 2}\left(\dfrac{1}{1-1/\left(2^{k-1}\log 2\right)}\right),
\end{align*}
where we have used the fact that $\log(1 - 2^{-k}) > -2^{1-k}$, which is valid since $0< 2^{-k} <1/2$ for all $k>200$. This further simplifies as
\begin{align*}
	n - 1 
	&< \frac{(k+1)\log 2 + k\log k - 2\log(k-1)}{\log 2}\left(1+\dfrac{1}{2^{k-2}\log 2}\right), 
\end{align*}
where the inequality $1/(1 - x) < 1 + 2x$, is valid for all $0<x<0.5$, with $x := 1/((2^{k-1} - 1)\log 2)$. Continuing in this way, we have
\begin{align*}
	n - 1 
	&< k + 1 + \frac{k \log k - 2\log(k-1)}{\log 2} 
	+ \frac{(k+1)\log 2 + k\log k}{2^{k-2} (\log 2)^2} \nonumber \\
	&< k + 1+ \frac{k \log k - 2\log(k-1)}{\log 2}+0.1,
\end{align*}
where we have used the fact that
\[
\frac{(k+1)\log 2 + k\log k}{2^{k-2} (\log 2)^2} < 0.1,
\]
for $k>200$. Since also
$$
\log(k-1)=\log k+\log(1-1/k)>\log k-2/k,
$$
we get that 
$$
n<k+2+\frac{(k-1)\log 2}{\log k}{\log 2}+0.1+\frac{2}{k(\log 2)}<k+\frac{(k-1)\log k}{\log 2}+2.3,
$$
since $k>200$. Next, we find a similar lower bound for $n$. We again combine \eqref{eq2.3}, \eqref{eq2.4} and \eqref{main} to have
\begin{align*}
	2^{n+1} > 2\alpha^{n} \ge L_n^{(k)} 
	= \dfrac{2^{k+1}k^k-(k+1)^{k+1}}{(k-1)^2}
	= \dfrac{2^{k+1}k^k}{(k-1)^2}\left(1-\dfrac{(k+1)^{k+1}}{2^{k+1}k^k(k-1)^2} \right).
\end{align*}
This means that
\begin{align*}
	2^{n+1} >  \dfrac{2^{k+1}k^k}{(k-1)^2}\left(1-\dfrac{(k+1)^{k+1}}{2^{k+1}k^k(k-1)^2} \right)
	>\dfrac{2^{k+1}k^k}{(k-1)^2}\left(1-\dfrac{4(k+1)}{2^{k+1}(k-1)^2} \right),
\end{align*}
since $(1 + 1/k)^k < 4$ holds for $k \geq 2$. Therefore, we have
\begin{align*}
	2^{n+1} 
	&>\dfrac{2^{k+1}k^k}{(k-1)^2}\left(1-\dfrac{1}{2^{k-1}} \right),
\end{align*}
because $ (k - 1)^2>k+1$ for $k >200$. Taking logarithms, we get
\begin{align*}
	n + 1 
	&> \frac{(k+1)\log 2 + k \log k - 2\log(k-1)}{\log 2} + \frac{\log(1 - 1/2^{k-1})}{\log 2} \nonumber \\
	&> k + 1 + \frac{k \log k - 2\log(k-1)}{\log 2} - \frac{1}{2^{k-2}\log 2},
\end{align*}
where we have used the fact that $\log(1 - 2^{1-k}) > -2^{2-k}$, which is valid since $0< 2^{1-k} <1/2$ for $k>200$. Therefore,
\begin{align*}
	n 
	&> k  + \frac{k \log k - 2\log(k-1)}{\log 2}-0.1,
\end{align*}
since $1/2^{k-2}\log 2 < 0.1$ for all $k >200$. Hence,
$$
n>k+\frac{(k-2)\log k}{\log 2}-0.1.
$$
Thus, as a result of the previous analysis, we have shown that
\begin{align}\label{bound_n}
	k  + \frac{(k-2)\log k}{\log 2}-0.1<n<k + \frac{(k-2)\log k}{\log 2}+2.3.
\end{align}
We record this as a lemma. 

\begin{lemma}
\label{lem:n}
If $n\ge 0$ and $k\ge 2$ are such that equation \eqref{main} holds, then $k>200$ and 
$$
k  + \frac{(k-2)\log k}{\log 2}-0.1<n<k + \frac{(k-2)\log k}{\log 2}+2.3.
$$
\end{lemma}

\subsection{Bounding $k$ and $n$}

In this section, we  find absolute bounds on $n$ and $k$. We prove the following result.
\begin{lemma}\label{lem1}
	Let $n\ge 0$ and $k>200$ be integer solutions to \eqref{main}. Then $k<7\cdot 10^{16}$ and $n<4\cdot 10^{18}$.
\end{lemma}
\begin{proof}
	Since $k>200$, then the upper bound on $n$ in \eqref{bound_n} gives
\begin{align*}
	n<k + \frac{(k-2)\log k}{\log 2}+2.3<2^{k/2}.
\end{align*}
This puts us in position to use inequality \eqref{lk_b1}, which together with \eqref{lk_b} gives
\begin{align*}
	\left|\dfrac{2^{k+1}k^k-(k+1)^{k+1}}{(k-1)^2}-3\cdot2^{n-2}\right| &=\left|L_n^{(k)} - 3\cdot2^{n-2}\right|\\
	&\le \left|L_n^{(k)} - f_k(\alpha)(2\alpha - 1)\alpha^{n - 1}\right|+\left|f_k(\alpha)(2\alpha - 1)\alpha^{n - 1} - 3\cdot2^{n-2}\right|\\
	&< \dfrac{3}{2} + 3\cdot 2^{n-2}\cdot \frac{36}{2^{k/2}}\\
	&< 2 +  2^{n}\cdot \frac{27}{2^{k/2}}\\
	&< 2^{n+1}\cdot \frac{27}{2^{k/2}},
\end{align*}
since by \eqref{bound_n}, $n>k  + (k-2)(\log k)/\log 2-0.1>k/2$, for all $k>200$. Therefore,
\begin{align*}
	\left|\dfrac{2^{k+1}k^k}{(k-1)^2}-3\cdot2^{n-2}\right| &< 2^{n+1}\cdot \frac{27}{2^{k/2}}+\dfrac{(k+1)^{k+1}}{(k-1)^2}.
\end{align*}
Denoting 
\begin{align*}
	A:=\dfrac{2^{k+1}k^k}{(k-1)^2} \qquad \text{and} \qquad B:=3\cdot2^{n-2},
\end{align*}
we then get that for some $\varepsilon\in \{\pm 1\}$ (according to which one is larger between $A$ and $B$), we have
\begin{align*}
	\left|\left((3(k-1)^2)^{-1}\cdot 2^{k-n+3}\cdot k^k \right)^\varepsilon-1\right| &< 2^{n+1}\cdot \frac{27}{2^{k/2}\max\{A,B\}}+\dfrac{(k+1)^{k+1}}{(k-1)^2\max\{A,B\}}\\
	&\le 2^{n+1}\cdot \frac{27}{2^{k/2}\cdot 3\cdot 2^{n-2}}+\dfrac{(k+1)^{k+1}}{(k-1)^2}\cdot \dfrac{(k-1)^2}{2^{k+1}k^k}\\
	&<  \frac{72}{2^{k/2}}+\dfrac{4(k+1)}{2^{k+1}}
	=  \frac{72}{2^{k/2}}+\dfrac{k+1}{2^{k-1}}\\
	&<\frac{72}{2^{k/2}}+\dfrac{1}{2^{k/2}}=\frac{73}{2^{k/2}},
\end{align*}
for all $k>200$. 

Let $\Gamma:=\left((3(k-1)^2)^{-1}\cdot 2^{k-n+3}\cdot k^k\right)^\varepsilon-1$. Then
\begin{align}\label{g1}
	\left|\Gamma\right|<\frac{73}{2^{k/2}}.
\end{align}
Since $k>200$, then there is a prime $p>3$ such that $p\mid k(k-1)$ and has a nonzero exponent in the factorization of
$3^{-1}\cdot 2^{k-n+3}\cdot k^k (k-1)^{-2}$, so $\Gamma\ne 0$. Indeed, for if not, we would have $k(k-1)=2^a\cdot 3^b$. Since $k-1>1$ and $k$ and $k-1$ are coprime, this gives 
$\{k,k-1\}=\{2^a,3^b\}$, so $3^b-2^a=\pm 1$ and the largest solution of this equation is $(a,b)=(3,2)$ by known results on Catalan's equation. This gives $k\le 9$ contradicting the fact that $k>200$. 
We  use Theorem \ref{thm:Mat} on the left-hand side of \eqref{g1}. Let
\begin{equation}\nonumber
	\begin{matrix}
		\gamma_{1}:=3(k-1)^2, & \gamma_{2}:=2, & \gamma_{3}:=k,\\
		b_{1}:=-\varepsilon, & b_{2}:=(k-n+3)\varepsilon, & b_{3}:=k\varepsilon.
	\end{matrix}
\end{equation}
We have $B:=n\ge \max\{|b_1|,|b_2|,|b_3|\}$, by the lower bound in Lemma \ref{lem:n}. Therefore, Theorem \ref{thm:Mat} gives
\begin{align}\label{g2}
	\log |\Gamma|&>-1.4\cdot 30^{3+3}\cdot 3^{4.5} (1+\log n)\cdot \log (3(k-1)^2) \cdot \log 2\cdot \log k.
\end{align}
Comparing \eqref{g1} with \eqref{g2} and using $3(k-1)^2<k^3$ valid for $k>200$, we get
\begin{align}\label{g3}
		(k/2)\log 2 - \log 73 &< 1.4\cdot 30^{3+3}\cdot 3^{4.5} (1+\log n)\cdot (3\log k)\log 2\cdot \log k\nonumber\\
		&<3\cdot 10^{11}\left(1+\dfrac{1}{\log n}\right)\log n \cdot(\log k)^2.
\end{align}
Notice that since $k>200$, then by \eqref{bound_n}, we always have that 
\begin{align*}
	1700<k  + \frac{(k-2)\log k}{\log 2}-0.1<n<k + \frac{(k-2)\log k}{\log 2}+2.3<3k\log k.
\end{align*}
So, using $k>200$ and $1700<n<3k\log k$, then \eqref{g3} becomes
\begin{align*}
	(k/2)\log 2 
	&<3\cdot 10^{11}\left(1+\dfrac{1}{\log 1700}\right)\log (3k\log k) \cdot(\log k)^2+\log 73\\
	&<3.5\cdot 10^{11}(\log k)^2\log(3k\log k).
\end{align*}
The above inequality gives $k<7\cdot 10^{16}$ and by Lemma \ref{lem:n} we get $n<4\cdot 10^{18}$, which is what we wanted.
\end{proof}

We write $n=r+m(k+1)$  with $r\in \{0,1,\ldots,k\}$. Recall that $k>200$ and 
$$
n>k+\frac{(k-2)\log k}{\log 2}-0.1,
$$
so that 
$$
m\ge \frac{(k-2)\log k/\log 2-0.1}{k+1},
$$
therefore $m\ge 9$. Also, since $k<7\times 10^{16}$ by Lemma \ref{lem:n}, we get $m\le 55$. We shall treat separately the cases $r=0$, $r\in \{1,2\}$ and $r\in \{3,\ldots,k\}$. 

\medskip

\subsection{The case $r=0$}

\medskip

In this case $L_n^{(k)}$ is even so $k$ is odd. Using Lemma \ref{lem:2adic} and Lemma \ref{lem:disc}, we get that 
$$
\pm 2\equiv 0\pmod {2^{k-2}},
$$
which is false for $k>4$. When $k=3$, the absolute value of the discriminant of $g_3(X)$ is $44$ and  does not  belong to $\{L_n^{(3)}\}_{n\ge 0}$. 

\medskip

\subsection{The cases $r=1,~2$}

\medskip

In this case $L_n^{(k)}$ is odd so $k$ is even. Reducing the equation \eqref{main} modulo $2^{k-1}$ and using the congruences \eqref{eq:r=1} and \eqref{eq:r=2} in Lemma \ref{lem:2adic} (ii) and (iii) respectively, we get 
$$
-(k+1)^{k+1}/(k-1)^2\equiv (-1)^m  (4m+1),~(-1)^m(4m^2+6m+3)\pmod {2^{k-1}}.
$$
In particular,
$$
2^{k-1}\mid \alpha_1^{b_1}-\alpha_2,
$$
where 
\begin{equation}
\label{eq:ABC}
\alpha_1:=k+1,\quad b_1:=k+1,\quad \alpha_2:=\left\{\begin{matrix} (-1)^m(k-1)^2(4m+1) & {\text{\rm if}} & r=1;\\ (-1)^m(k-1)^2(4m^2+6m+3) & {\text{\rm if}} & r=2.\end{matrix}\right.
\end{equation}
In particular,
$$
k-1\le \nu_2(\alpha_1^{b_1}-\alpha_2).
$$
To find an upper bound on the right--hand side above, we use Lemma \ref{lem:BL}. It is easy to see that $\alpha_1,~\alpha_2$ are multiplicatively independent. 
Indeed, if not, every prime factor of $\alpha_2$ should also be a prime factor of $\alpha_1$. Hence, if $p\mid k-1$ is prime then also $p\mid k+1$. In particular $p=2$, which makes $k-1$ a power of $2$, which is false since $k>200$ is even. 
We take $p:=2,~\log A_1:=\log(k+1)=\log\alpha_1$. Since $m<3k\log k$, it follows that 
$$
4m^2+6m+3\le 4(3k\log k)^2\left(1+\frac{1}{2k\log k}+\frac{3}{4(3k\log k)^2}\right)<36.1(k\log k)^2<k^{3.4}\quad {\text{\rm for}}\quad k> 200. 
$$
Thus, we can take $\log A_2:=5.4\log k>\log \alpha_2$. We take 
$$
b':=\frac{k+6.4}{5.4\log k}>\frac{k+1}{5.4\log k}+\frac{1}{\log(k+1)}=\frac{b_1}{\log A_2}+\frac{b_2}{\log A_1}.
$$
We take $B:=\max\{\log(b'+\log\log 2+0.4,10\log 2\}.$  Note also that $g=1$ since both $\alpha_1$ and $\alpha_2$ are odd. Then Lemma \ref{lem:BL} tells us that 
$$
\nu_2(\alpha_1^{b_1}-\alpha_2)\le \frac{24\cdot 2}{(\log 2)^4} B^2 \log(k+1)\times (5.4\log k)<1123 B^2(\log k)\log(k+1).
$$
We thus get
$$
k-1<1123 B^2 (\log k)\log(k+1).
$$
When $B=10\log 2$, we get  $\log b'+\log\log 2+0.4\le 10\log 2$. This leads to $k<59000$. When 
$$
B=\log b'+\log\log 2+0.4,
$$ 
we get
$$
k-1\le 1123 \left(\log\left(\frac{k+6.4}{5.4\log k}\right)+\log\log 2+0.4\right)^2(\log k)\log(k+1),
$$
which gives $k<7\times 10^7$. Let us record what we have proved. 

\begin{lemma}
\label{lem:ris12}
When $r=1,2$, then $k$ is odd and $200<k<7\times 10^{7}$. Furthermore, 
\begin{equation}
\label{eq:cong}
2^{k-1}\mid \alpha_1^{b_1}-\alpha_2,
\end{equation}
where $\alpha_1,~b_1,~\alpha_2$ are shown in \eqref{eq:ABC}.
\end{lemma}

A short computation finished the proof of this case. Namely, for each even $k$ in the interval $(200, 70000000)$, we looked at the positive integers $n$  in the interval 
$$
(k+(k-2)\log k/\log 2-0.1, k+(k-2)\log k/\log 2+2.3)
$$ 
(Lemma \ref{lem:n}), and tested whether there is such an integer $n$ in the above interval which is congruent to 
$r\in \{1,2\}$ modulo $k+1$. If such an $n$ passed this test, we recorded the pair $(k,n)$ in some list denoted by ${\mathcal A}$. The code ran for a few minutes and produced a list  ${\mathcal A}$ of $32$ pairs $(k,n)$. 
For each of these $32$ pairs we computed $m:=\lfloor n/(k+1)\rfloor$ and tested whether
$$
(k+1)^{k+1}-\alpha_2
$$
was a multiple of $2^{100}$, where $\alpha_2:=(-1)^m (k-1)^2(4m+1)$ or $\alpha_2:=(-1)^m(k-1)^2(4m^2+6m+3)$ according to whether $r=1$ or $2$, respectively (see \eqref{eq:cong}). None of the $32$ pairs passed this last test, see Appendix \ref{app1}. We also used the \texttt{pow(k+1,k+1,mod)} to handle large numbers of the form $(k+1)^{k+1}$.
This shows that there are no solutions with $r\equiv 1,2\pmod {k+1}$. 

\medskip 

\subsection{The case $r\in \{3,\ldots,k\}$}

\medskip

In this case $L_n^{(k)}$ is even so $k$ is odd. We take a look at the congruence \eqref{eq:rge3} of Lemma \ref{lem:2adic} (iv). Recall that
$$
L_n^{(k)}\equiv 2^{r-2} (-1)^m L(m,r)\pmod {2^{k+r-2}},
$$
where
$$
L(m,r):=4\left(\binom{m+r+1}{m}-\binom{m+r-1}{m-2}\right)-\left(\binom{m+r}{m}-\binom{m+r-2}{m-2}\right).
$$
Note that 
\begin{eqnarray*}
L(m,r) & = & \frac{4(m+r+1)!}{m!(r+1)!}-\frac{4(m+r-1)!}{(m-2)!(r+1)!}-\frac{(m+r)!}{m! r!}+\frac{(m+r-2)!}{(m-2)! r!}\\ 
& = & \frac{(m+r-2)!}{(m-2)! r!}\left(\frac{4(m+r+1)(m+r)(m+r-1)}{m(m-1)(r+1)}-\frac{4(m+r-1)}{r+1}-\frac{(m+r)(m+r-1)}{m(m-1)}+1\right)\\
& = & \frac{1}{m(m-1)(r+1)} \binom{m+r-2}{m-2}\left(3r^3+10mr^2+8m^2r+2mr-3r+8m^2-8m\right).
\end{eqnarray*}
Thus,
$$
\nu_2(L(m,r))\le \nu_2\left(\binom{m+r-2}{m-2}\right)+\nu_2(3r^3+10mr^2+8m^2r+2mr-3r+8m^2-8m).
$$
Note that 
$$m(k+1)<r+m(k+1)=n<k+\frac{(k-2)\log k}{\log 2}+2.3<3k\log k,
$$
therefore $m<3\log k$. Also, $r\le k$. Thus,
\begin{eqnarray*}
&& 3r^3+10mr^2+8m^2r+2mr-3r+8m^2-8m\\
& \le &  3k^3+10k^2(3\log k)+8k(3\log k)^2+2k(3\log k)+8(3\log k)^2\\
& \le & 3k^3\left(1+\frac{10\log k}{k}+\frac{24(\log k)^2}{k^2}+\frac{2\log k}{k^2}+\frac{24 (\log k)^2}{k^3}\right)\\
& < & 4k^3.
\end{eqnarray*}
The above bounds hold since $k>200$. Thus, 
$$
\nu_2(L(m,r))\le \nu_2\left(\binom{m+r-2}{m-2}\right)+\log_2(4k^3)\le \log_2(m+r-1)+\log_2(4k^3),
$$
where in the above we applied Kummer's theorem (see \cite{kum}) to bound the exponent of $2$ in the binomial coefficient. Since $m+r-1<k+3\log k-1<2k$ for $k>200$, we get that 
$$
\nu_2(L(m,r))\le \log_2(2k)+\log_2(4k^3)=3+4\log_2(k)=3+\frac{4\log k}{\log 2}<6\log k+2,
$$
since $k>200$. Thus, if we write 
\begin{equation}
\label{eq:Lmr}
L(m,k)=2^a\cdot b,
\end{equation}
where $a\ge 0$ and $b$ is odd, then $0\le a\le 6\log k+2$. 
Thus, we get that 
$$
L_n^{(k)}\equiv 2^{r-1} L(m,k)\pmod {2^{r+k-2}}\equiv 2^{r-1+a} b\pmod {2^{r+k-2}}.
$$
Note that 
$$
r-1+a<r+6\log k+1<r+k-2,
$$
where the last inequality is equivalent to $k>6\log k+3$ which holds for us since $k>200$. Thus, the above argument shows that 
$$
\nu_2(L_n^{(k)})=r-1+a.
$$
Comparing this with Lemma \ref{lem:disc}, we get that
$$
r-1+a=\nu_2(L_n^{(k)})=\nu_2({\text{\rm Disc}}(g_k(X)))=k-1,
$$
so $r=k-(a-1)$. Since $r\le k$, we get that $a-1\ge 0$. Furthermore, since $a\le 6\log k+2$ and $k<7\times 10^{16}$, we get that $a-1\le 233$. 

We record what we have proved. 

\begin{lemma}
In case $r\in \{3,4,\ldots,k\}$, then $k$ is odd and $r=k-(a-1)$, where $a-1\in [0,233]$.  
\end{lemma}

So, now we know that $n=r+m(k+1)$, where $r=k-(a-1)$, $0\le a-1\le 6\log k+2$ and $m<3\log k$. We next exploit relation \eqref{g1}. Writing it in logarithmic form we get
$$
|k\log k-(n-k-3)\log 2-\log(3(k-1)^2)|<\frac{146}{2^{k/2}}.
$$
We rework the left--hand side above as
\begin{eqnarray*}
k\log k-(n-k-3)\log 2-\log(3(k-1)^2) & = & k\log k-(k-(a-1)+m(k+1)-k-3)\log 2-\log(3(k-1)^2)\\
& = & k\log(k/2^{m})-(m-a-2)\log 2-\log(3(k-1)^2).
\end{eqnarray*}
So, we get that 
$$
|k\log(k/2^{m})|<\frac{146}{2^{k/2}}+|m-a-2|\log 2+\log(3(k-1)^2).
$$
The right--hand side is at most
$$
\frac{146}{2^{k/2}}+\max\{m,a+2-m\}\log 2+\log(3(k-1)^2)\le \frac{146}{2^{k/2}}+(6\log k+4-9)\log 2+\log 3+2\log k<7\log k.
$$
In the above, we used that $9\le m\le 3\log k<6\log k+4-9$ and $a<6\log k+2$. Hence,
$$
\left|\log(2^{m}/k)\right|<\frac{7\log k}{k},
$$ 
which implies that 
$$
\exp\left(-\frac{7\log k}{k}\right)<\frac{2^{m}}{k}<\exp\left(\frac{7\log k}{k}\right).
$$
Let 
$$
\zeta:=\frac{7\log k}{k}.
$$
We have $\zeta<0.2$ since $k>200$ and the function $(7\log x)/x$ is decreasing for $x>200$. Hence,
$$
\exp(\zeta)=1+\zeta+\eta,
$$
where 
$$
\eta=\frac{\zeta^2}{2!}+\frac{\zeta^3}{3!}+\cdots. 
$$
Thus, 
$$
|\eta|\le \frac{|\zeta|^2}{2} (1+|\zeta|+\cdots+)=\frac{\zeta^2}{2(1-|\zeta|)}<\frac{\zeta^2}{1.6}=\frac{(7\log k)^2}{1.6k^2}.
$$
So, we get that 
$$
\exp(\zeta) <1+\frac{7\log k}{k}+\frac{(7\log k)^2}{1.6k^2}\qquad {\text{\rm and}}\qquad \exp(-\zeta)>1-\frac{7\log k}{k}-\frac{(7\log k)^2}{1.6k^2}.
$$
Hence,
$$
1-\frac{7\log k}{k}-\frac{(7\log k)^2}{1.6k^2}<\frac{2^{m}}{k}<1+\frac{7\log k}{k}-\frac{(7\log k)^2}{1.6k^2},
$$
which is equivalent to 
$$
k-(7\log k)-\frac{(7\log k)^2}{1.6k}<2^{m}<k+(7\log k)+\frac{(7\log k)^2}{1.6k}.
$$
Since $k>200$, we have that 
$$
\frac{(7\log k)^2}{1.6 k}<5,
$$
and since $k<7\times 10^{16}$, we also have that $7\log k<274$. Therefore,
$$
|k-2^{m}|<300,
$$
which leads to $k\in (2^{m}-300,2^{m}+300)$. In conclusion, we must have 
$$
n=r-(a-1)+m(k+1),\quad {\text{\rm where}}\quad a-1\in [0,233],\quad m\in [9,55],\quad k\in (2^{m}-300,2^{m}+300),
$$
and $k$ is odd. Thus, there are at most  $234\times 48\times 300<3.5\times 10^6$  ways to fix the triple $(a,m,k)$. This fixes the pair $(k,n)$ in at most $3.5\times 10^6$ ways. 
For each of these we test, by the definition of $a$, whether
$$
\nu_2(L(m,r))=a,
$$
where $L(m,r)$ is given by formula \eqref{eq:Lmr}. If the pair $(k,n)$ passes this test we still use Lemma \ref{lem:2adic} (iv) to check whether
$$
L_n^{(k)}\equiv 2^{r-2} (-1)^m L(m,r)\pmod {2^{r+k-2}} \equiv \frac{2^{k+1}k^k-(k+1)^{k+1}}{(k-1)^2}\pmod {2^{r+k-2}}.
$$
This is implied by 
$$
(-1)^m (k-1)^2 L(m,r)- (2^{k+3-r} k^k-2^{k+3-r}((k+1)/2)^{k+1} )\equiv 0\pmod {2^{k}}.
$$
Since $k+3-r=a+2$ and since $k>(\log 6k+4)+150\ge a+150$ for $k>200$, we checked the above congruence just modulo $2^{a+150}$. In a matter of seconds this was checked and no solution was found. 
Interesting enough when we only checked modulo $2^{a+100}$, then $14$ pairs $(k,n)$ were found. The theorem is therefore proved.

\section*{Acknowledgments} 
The first author thanks the Mathematics division of Stellenbosch University for funding his PhD studies.

\section*{Addresses}

$ ^{1} $ Mathematics Division, Stellenbosch University, Stellenbosch, South Africa.

Email: \url{hbatte91@gmail.com}

Email: \url{fluca@sun.ac.za}

\appendix
\section{Appendix 1}\label{app1}
\begin{verbatim}
import math
def search_and_print_pairs_with_r(start_k=202, end_k=70000000):
"""Search for all (k, n) pairs with n \equiv 1 or 2 mod (k+1), and print (k, n, r)."""
        results = []
        log2 = math.log(2)
	
        for k in range(start_k, end_k + 1, 2):  # even k only
            logk = math.log(k)
            f_k = k + ((k - 2) * logk / log2) - 0.1
            g_k = k + ((k - 2) * logk / log2) + 2.3
        
            # Only include integers strictly inside the open interval (f(k), g(k))
            lower = int(math.floor(f_k)) + 1
            upper = int(math.floor(g_k))
	
           for n in range(lower, upper + 1):
              r = n % (k + 1)
              if r in {1, 2}:
                 results.append((k, n, r))
                 print(f"(k={k}, n={n}, r={r})")  # Print immediately
	
         print(f"\nTotal number of (k, n, r) triples found: {len(results)}")
         return results
	
def test_power_of_two_congruence(triples):
"""Tests whether (k+1)^{k+1} - alpha_2 \equiv 0 mod 2^100 for each (k, n, r) triple."""
    results = []
    mod = 2**100
	
    for k, n, r in triples:
         m = n // (k + 1)
	     if r == 1:
             alpha2 = (-1)*m * (k - 1)*2 * (4 * m + 1)
         elif r == 2:
             alpha2 = (-1)*m * (k - 1)2 * (4 * m*2 + 6 * m + 1)
         else:
             continue  # Invalid r (shouldn't happen)
	
         lhs = pow(k + 1, k + 1, mod)
         alpha2_mod =alpha2 % mod
         if (lhs -alpha2_mod) % mod == 0:
             results.append((k, n, r, m))
             print(f"(k={k}, n={n}, r={r}, m={m}) passes the test.")
	print(f"\nTotal number of valid (k, n, r, m) quadruples: {len(results)}")
    return results
	
# --- Run both steps ---
A = search_and_print_pairs_with_r(200, 70000000)  # Step 1: Generate triples
valid_quads = test_power_of_two_congruence(A)     # Step 2: Test congruence
\end{verbatim}

\end{document}